\documentclass[12pt, reqno]{amsart}
\allowdisplaybreaks
\begin{document}
\setcounter{page}{1}
\newcommand{\keyw}[1]{\par\noindent{\bf Keywords: }#1.}
\numberwithin{equation}{section}
\numberwithin{figure}{section}
\newtheorem{theorem}{Theorem}[section]
\newtheorem{lemma}[theorem]{Lemma}
\newtheorem{corollary}[theorem]{Corollary}
\newtheorem{proposition}[theorem]{Proposition}
\theoremstyle{definition}
\newtheorem{definition}[theorem]{Definition}
\newtheorem{example}[theorem]{Example}
\theoremstyle{remark}
\newtheorem{remark}[theorem]{Remark}
\newtheorem{conjecture}[theorem]{Conjecture}
\newtheorem{notation}[theorem]{Notation}
\newtheorem{openproblem}[theorem]{Open Problem}

\title[\hfilneg \hfil Fractional Difference Equation]
{ Existence of Some Positive Solutions to Fractional Difference Equation}

\author[Deepak B. Pachpatte, Arif S. Bagwan, Amol D. Khandagal]
{Deepak B. Pachpatte, Arif S. Bagwan, Amol D. Khandagale}

\address{Deepak B. Pachpatte \newline
 Department of Mathematics,
 Dr. Babasaheb Ambedkar Marathwada University, Aurangabad,
 Maharashtra 431004, India}
\email{pachpatte@gmail.com}

\address{Arif S. Bagwan \newline
  Department of First Year Engineering,
  Pimpri Chinchwad College of Engineering, Nigdi, Pune,
 Maharashtra  4110444, India}
\email{arif.bagwan@gmail.com}

\address{Amol D. Khandagale \newline
 Department of Mathematics,
 Dr. Babasaheb Ambedkar Marathwada University, Aurangabad,
 Maharashtra 431004, India}
\email{ kamoldsk@gmail.com}

\subjclass[2010]{39A10, 26A33}
\keywords{  Fractional difference equation, existence, positive solutio}

\begin{abstract}
 The main objective of this paper is to study the existence of solutions to some basic fractional difference equations. The tools employed are Krasnosel'skii fixed point theorem which guarantee at least two positive solutions.
\end{abstract}

\maketitle

\section{Introduction}

The theory of fractional calculus and associated fractional differential equations in continuous case has received great attention. However, very limited progress has been done in the development of the theory of finite fractional difference equations. But, recently a remarkable research work has been made in the theory of fractional difference equations. Diaz and Osler \cite{10} introduced a discrete fractional difference operator defined as an infinite series. \\

Recently, a variety of results on discrete fractional calculus have been published by Atici and Eloe \cite{4,5,7} with delta operator. Atici and Sengul \cite{6} provided some initial attempts by using the discrete fractional difference equations to model tumor growth. M. Holm \cite{11} extended his contribution to discrete fractional calculus by presenting a brief theory for composition of fractional sum and difference. Furthermore, Goodrich \cite{1,2,3} developed some results on discrete fractional calculus in which he used Krasnosel'skii fixed point theorem to prove the existence of initial and boundary value problems. Following this trend, H. Chen, et. al. \cite{13} and S. Kang, et. al. \cite{14} discussed about the positive solutions of BVPs of fractional difference equations depending on parameters. H. Chen, et. al. \cite{8}, in their article provided multiple solutions to fractional difference boundary value problems using various fixed point theorems.\\

In this paper, we consider the boundary value problems of fractional difference equation of the form,
\[
 -\Delta^{v}y(t)=\lambda h(t+v-1)f(t+v-1,y(t+v-1)),
\tag{1.1}\]
\[
  y(v-2)=y(v+b)=0.
\tag{1.2}\]
Where, $ t\in [0,b]_{\mathbb{N}_{0}} $, $ f:[v-1,v+b]_{\mathbb{N}_{v-1}}\times \mathbb{R} \rightarrow \mathbb{R}  $ is continuous, $ h:[v-1,v+b]_{\mathbb{N}_{v-1}}\longrightarrow [0,\infty) $, $ 1<v\leq 2 $ and $\lambda$  is a positive parameter. \\

The present paper is organized as follows. In section 2, together with some basic definitions, we will demonstrate some important lemmas and theorem in order to prove our main result. In section 3, we establish the results for existence of solutions to the boundary value problem $ (1.1)-(1.2) $ using Krasnosel'skii fixed point theorem.

\section{Preliminaries}

In this section, let us first collect some basic definitions and lemmas that are very much important to us in the sequel.\\

\begin{definition}\cite{3,7}\label{} We define,
\[t^{\underline{v}}=\frac{\Gamma(t+1)}{\Gamma(t+1-v)},
\tag{2.1} \]
 for any $t$ and $v$ for which right hand side is defined. We also appeal to the convention that if $t+v-1$ is a pole of the Gamma function and $t+1$ is not a pole, then $t^{\underline{v}}=0$.
\end{definition}
\begin{definition}\cite{3,7}\label{}
The $ v^{th} $ fractional sum of a function $f$, for $v>0$ is defined as,
\[
\Delta^{-v}f(t)=\Delta^{-v}f(t,a):=\frac{1}{\Gamma(v)}\sum_{s=a}^{t-v}(t-s-1)^{(v-1)}f(s),
\tag{2.2} \]
for $ t\in\left\lbrace a+v,a+v+1,\dots\right\rbrace =:\mathbb{N}_{a+v} $. We also define the $ v^{th} $ fractional difference for $ t^{\underline{v}}=0 $ by $\Delta^{v}_{a+v}f(t):=\Delta^{N}\Delta^{v-N}f(t)$, where $ t\in \mathbb{N}_{a+v} $ and $ N\in \mathbb{N} $ is chosen so that $ 0\leq N-1< v\leq N $.
\end{definition}
Now we give some important lemmas.\\

\begin{lemma}\cite{3,7}\label{2.1}
Let $t$ and $v$ be any numbers for which $t^{\underline{v}}$ and $ t^{\underline{v-1}} $ are defined. Then $\Delta t^{\underline{v}}=vt^{\underline{v-1}}$. \\
\end{lemma}
\begin{lemma}\cite{3,7}\label{2.2}
Let  $ 0\leq N-1< v\leq N $. Then
\[
\Delta^{-v}\Delta^{v}y(t)=y(t)+C_{1}t^{\underline{v-1}}+C_{2}t^{\underline{v-2}}+\dots +C_{N}t^{\underline{v-N}},
\tag{2.3} \]
for some $ C_{i}\in R $ with $ 1\leq i\leq N $.\\
\end{lemma}

\begin{lemma}\cite{7}\label{2.3}
Let  $ 1< v\leq 2 $ and $ f:[v-1,v+b]_{\mathbb{N}_{v-1}}\times \mathbb{R} \rightarrow \mathbb{R} $ be given. Then the solution of fractional boundary value problem $-\Delta^{v}y(t)=f(t+v-1,y(t+v-1))$, \,\,\,\,\,\,
$ y(v-2)=y(v+b+1)=0 $
 is given by
\[
 y(t)=\sum_{s=0}^{b+1}{G(t,s)f(s+v-1,y(s+v-1))},
\tag{2.4} \]
 where Green’s function $ G:[v-1,v+b]_{\mathbb{N}_{v-1}}\times[0,b+1]_{\mathbb{N}_{0}}\rightarrow \mathbb{R} $ is defined by
\[
G\left( {t,s} \right) = \frac{1}{{\Gamma v}}\left\{ \begin{array}{l}
\frac{{{t^{\underline {v - 1} }}{{\left( {v + b - s} \right)}^{\underline {v - 1} }}}}{{{{\left( {v + b - 1} \right)}^{\underline {v - 1} }}}} - {\left( {t - s - 1} \right)^{\underline {v - 1} }},\,\,\,\,\,\,0 \le s < t - v + 1 \le b + 1\\
\frac{{{t^{\underline {v - 1} }}{{\left( {v + b - s} \right)}^{\underline {v - 1} }}}}{{{{\left( {v + b - 1} \right)}^{\underline {v - 1} }}}},\,\,\,\,\,\,\,\,\,\,\,\,\,\,\,\,\,\,\,\,\,\,\,\,\,\,\,\,\,\,\,\,\,\,\,\,\,\,\,\,\,\,\,\,\,\,\,\,0 \le t - v + 1 < s \le b + 1,
\end{array} \right.
\tag{2.5} \]
\end{lemma}

\begin{lemma}\cite{7}\label{}
The Green’s function $ G(t,s) $ given in above lemma satisfies,
\begin{enumerate}
\item  $ G\left( {t,s} \right) \ge 0 $ for each $\left( {t,s} \right) \in {\left[ {v - 2,v + b} \right]_{\mathbb{N}_{{_{v - 2}}}}} \times {\left[ {0,\,b + 1} \right]_{\mathbb{N}{_0}}}$
\item $ \mathop {\max }\limits_{t \in {{\left[ {v - 2,v + b} \right]}_{\mathbb{N}_{{_{v - 2}}}}}} G\left( {t,s} \right) = G\left( {s + v - 1,s} \right) $ for each $ s \in {\left[ {0,b} \right]_{\mathbb{N}{_0}}} $ and
\item There exists a number $ \gamma  \in \left( {0,1} \right) $ such that \\ $ \mathop {\min }\limits_{t \in {{\left[ {\frac{{v + b}}{4},\frac{{3(v + b)}}{4}} \right]}_{\mathbb{N}_{{_{v - 2}}}}}} G\left( {t,s} \right) \ge \mathop {\max }\limits_{t \in {{\left[ {v - 2,v + b} \right]}_{\mathbb{N}_{{_{v - 2}}}}}} G\left( {t,s} \right) = \gamma G\left( {s + v - 1,s} \right) $, for $ s \in {\left[ {0,b} \right]_{\mathbb{N}{_0}}} $

\end{enumerate}
\end{lemma}
Now we give the solution of fractional boundary value problem $ (1.1)-(1.2) $, if it exists.\\

\begin{theorem}\label{}
Let $ f:{\left[ {v - 1,v + b} \right]_{\mathbb{N}_{{_{v - 1}}}}} \times  \mathbb{R} \to \mathbb{R} $ be given. A function $ y(t) $ is a solution to discrete fractional boundary value problem $ (1.1)-(1.2) $  iff is a fixed point of the operator
\[
  Fy(t) = \lambda \sum\limits_{s = 0}^b {G\left( {t,s} \right)h\left( {s + v - 1} \right)f\left( {s + v - 1,y\left( {s + v - 1} \right)} \right)} ,
\tag{2.6} \]
where $ G(t,s) $ is given in above lemma $ (2.3) $ \\
\end{theorem}
\begin{proof} From lemma $(2.2)$ we find that a general solution to problem $ (1.1)-(1.2) $ \[y\left( t \right) =  - {\Delta ^{ - v}}\lambda h(t + v - 1)f(t + v - 1,y(t + v - 1)) + {C_1}{t^{\underline {v - 1} }} + {C_2}{t^{\underline {v - 2} }},\]
 from the boundary condition \(y(v-2)=0\),
\begin{align*}
y\left( {v - 2} \right) & = {\left. { - {\Delta ^{ - v}}\lambda h(t + v - 1)f(t + v - 1,y(t + v - 1))} \right|_{t = v - 2}} \\
& + {C_1}{(v - 2)^{\underline {v - 1} }} + {C_2}{(v - 2)^{\underline {v - 2} }}\\
& = {\left. { - \frac{1}{{\Gamma v}}\sum\limits_{s = 0}^{t - v} {{{(t - s - 1)}^{\underline {v - 1} }}} \lambda h(s + v - 1)f(t + v - 1,y(t + v - 1))} \right|_{t = v - 2}} \\
& + {C_2}\Gamma (v - 1)\\
& = {C_2}\Gamma (v - 1)\\
& =  0,
\end{align*}
therefore, $  C_2 = 0. $  \\
On the other hand, using boundary condition \(y(v+b)=0\)
\begin{align*}
y\left( {v + b} \right) & = {\left. { - {\Delta ^{ - v}}\lambda h(t + v - 1)f(t + v - 1,y(t + v - 1))} \right|_{t = v + b}} \\
& + {C_1}{(v + b)^{\underline {v - 1} }} + {C_2}{(v + b)^{\underline {v - 2} }}\\
& = {\left. { - \frac{1}{{\Gamma v}}\sum\limits_{s = 0}^{t - v} {{{(t - s - 1)}^{\underline {v - 1} }}} \lambda h(s + v - 1)f(t + v - 1,y(t + v - 1))} \right|_{t = v + b}} \\
& + {C_1}{(v + b)^{\underline {v - 1} }} \\
& = 0,
\end{align*}
\[{C_1}{(v + b)^{\underline {v - 1} }} = \frac{1}{{\Gamma v}}\sum\limits_{s = 0}^{t - v} {{{\left( {t - s - 1} \right)}^{\underline {v - 1} }}\lambda h(s + v - 1){{\left. {f(s + v - 1,y(s + v - 1))} \right|}_{t = v + b}}} \]
\[{C_1} = \frac{1}{{\Gamma v{{(v + b)}^{\underline {v - 1} }}}}\sum\limits_{s = 0}^b {{{\left( {v + b - s - 1} \right)}^{\underline {v - 1} }}\lambda h(s + v - 1)f(s + v - 1,y(s + v - 1))} \] \\
Using $ C_{1}$ and $ C_{2} $ in $ y(t) $,  we get
\[y(t) =  - \frac{1}{{\Gamma v}}\sum\limits_{s = 0}^{t - v} {{{(t - s - 1)}^{\underline {v - 1} }}} \lambda h(s + v - 1)f(s + v - 1,y(s + v - 1)) \]
 \[+ \frac{{{t^{\underline {v - 1} }}}}{{\Gamma v{{(v + b)}^{\underline {v - 1} }}}}\sum\limits_{s = 0}^b {{{\left( {v + b - s - 1} \right)}^{\underline {v - 1} }}\lambda h(s + v - 1)f(s + v - 1,y(s + v - 1))} \]
\begin{align*}
y(t) & = \sum\limits_{s = 0}^{t - v} {\left\{ {\frac{{{t^{\underline {v - 1} }}{{\left( {v + b - s - 1} \right)}^{\underline {v - 1} }}}}{{\Gamma v{{(v + b)}^{\underline {v - 1} }}}} - \frac{{{{(t - s - 1)}^{\underline {v - 1} }}}}{{\Gamma v}}} \right\}} \lambda h(s + v - 1)\\
& \cdot f(s + v - 1,y(s + v - 1))\\
& + \sum\limits_{s = t - v + 1}^b {\frac{{{t^{\underline {v - 1} }}{{\left( {v + b - s - 1} \right)}^{\underline {v - 1} }}}}{{\Gamma v{{(v + b)}^{\underline {v - 1} }}}}} \lambda h(s + v - 1)f(s + v - 1,y(s + v - 1))\\
\end{align*}
\[y(t) = \sum\limits_{s = 0}^b {G(t,s)} \lambda h(s + v - 1)f(s + v - 1,y(s + v - 1)),
\tag{2.7} \] \\
Consequently, we observe that $ y(t) $ implies that whenever $ y $ is a solution of $ (1.1)-(1.2) $, $ y $  is a fixed point of $ (2.6) $, as desired.\\
\end{proof}
\begin{theorem}\cite{12}\label{}
Let $ E $ be a banach space, and let $ \mathcal{K}\subset E $ be a cone in $ E $. Assume that $ \Omega_{1} $ and $ \Omega_{2} $ are open sets contained in $ E $ s. t. $ 0 \in \Omega_{1} $ and $ \overline{\Omega}_{1}\subseteq \Omega_{2} $, and let $ S:\mathcal{K}\cap{(\overline{\Omega}_{2}\setminus\Omega_{1})}\rightarrow \mathcal{K}$ be a completely continuous operator such that either
\begin{enumerate}
\item $ \left\| {Sy} \right\| \le \left\| y \right\| $ for $ y \in \mathcal{K} \cap \partial {\Omega _1} $ and $ \left\| {Sy} \right\| \ge \left\| y \right\| $ for $ y \in \mathcal{K} \cap \partial {\Omega _2} $; Or
\item $ \left\| {Sy} \right\| \ge \left\| y \right\| $ for $ y \in \mathcal{K} \cap \partial {\Omega _1} $ and $ \left\| {Sy} \right\| \le \left\| y \right\| $ for $ y \in \mathcal{K} \cap \partial {\Omega _2} $
\end{enumerate}
Then $ S $ has at least one fixed point in $ \mathcal{K}\cap{(\overline{\Omega}_{2}\setminus\Omega_{1})} $.\\
\end{theorem}

\section{Main Result}

To prove our main result let us state all required theorems for the existence of positive solutions to problem $ (1.1)-(1.2) $ \\
For this, let \[\eta : = \frac{1}{{\sum\limits_{s = 0}^b {G(s + v - 1,s)h(s + v - 1)} }},\]
\[\sigma : = \frac{1}{\gamma{\sum\nolimits_{s = \left[ {\frac{{b + v}}{4} - v + 1} \right]}^{\left[ {\frac{{3(b + v)}}{4} - v + 1} \right]} {G\left( {\left[ {\frac{{b - v}}{2}} \right] + v,s} \right)} }},\]
where $ \eta $ and $ \sigma $ are well defined by lemma $ 2.4 $ and $ \gamma $ is the number given by lemma $ 2.4.(3). $ \\
In the sequel, we present some conditions on $ f $ that will imply the existance of positive solutions. \\
$ H1:$ $ \exists $  a number $ r>0 $ such that $ f(t,y)\leq\frac{\eta r}{\lambda}$ whenever $ 0\leq y\leq r $.\\
$ H2: $ $ \exists $  a number $ r>0 $ such that $ f(t,y)\geq\frac{\sigma r}{\lambda}$ whenever $ \gamma r\leq y\leq r $. \\
$ H3: $ $ \mathop {\lim }\limits_{y \to {0^ + }} \mathop {\min }\limits_{t \in {{\left[ {v - 2,v + b} \right]}_{{\mathbb{N}_{v - 2}}}}} \frac{{f(t,y)}}{y} =  + \infty $. \\
$ H4: $ $ \mathop {\lim }\limits_{y \to {\infty ^ + }} \mathop {\min }\limits_{t \in {{\left[ {v - 2,v + b} \right]}_{{\mathbb{N}_{v - 2}}}}} \frac{{f(t,y)}}{y} =  + \infty $. \\
For our purpose, let $ E $ be a Banach space defined by
\[
 E = \left\{ {y:{{[v - 2,v + b]}_{{\mathbb{N}_{v - 2}}}} \to \mathbb{R},\,\,\,y(v - 2) = y(v + b)= 0} \right\},
\tag{3.1} \]
with norm, $ \left\| y \right\| = \max \left| {y(t)} \right|, \,\,\,\,\,\,\,\,\,\,\,\,\,\,\,\,t \in {[v - 2,v + b]_{{\mathbb{N}_0}}} $. \\
Also, define the cones
\[
{\mathcal{K}_0} = \left\{ {y \in E:0 \le y(t),\,\,\mathop {min}\limits_{t \in \left[ {\frac{{v + b}}{4},\frac{{3(v + b)}}{4}} \right]} y(t) \ge \gamma \left\| {y(t)} \right\|} \right\}.
\tag{3.2} \]
In order to prove our first existence, let us prove the following important lemma.

\begin{lemma}\label{}
$ F({\mathcal{K}_0}) \subseteq {\mathcal{K}_0} $  i.e., $ F $ leaves the cone $ \mathcal{K}_0 $ invariant.
\end{lemma}
\begin{proof} Observe that
\begin{align*}
\mathop {min}\limits_{t \in \left[ {\frac{{v + b}}{4},\frac{{3(v + b)}}{4}} \right]} (Fy)(t) & = \mathop {min}\limits_{t \in \left[ {\frac{{v + b}}{4},\frac{{3(v + b)}}{4}} \right]} \sum\limits_{s = 0}^b {G(t,s)\lambda h(s + v - 1)} \\
& \times f(s + v - 1,y(s + v - 1))\\
& \ge \gamma \sum\limits_{s = 0}^b {G(t,s)\lambda h(s + v - 1)f(s + v - 1,y(s + v - 1))} \\
& \ge \gamma \mathop {\max }\limits_{t \in {{\left[ {v - 2,v + b} \right]}_{{N_{v - 2}}}}} \sum\limits_{s = 0}^b {G(t,s)\lambda h(s + v - 1)} \\
& \times f(s + v - 1,y(s + v - 1))\\
& = \gamma \left\| {Fy} \right\|,
\end{align*}
which implies $ Fy \in {\mathcal{K}_0} $.
\end{proof}

\begin{theorem}\label{}
Assume that $ \exists $ distinct numbers $ r_1>0 $ and $ r_2>0 $ with $ r_1< r_2 $ such that $ f $ satisfies the condition $ H1 $ at $ r_1 $ and $ H2 $ at $ r_2 $. Then the fractional boundary value problem $ (1.1)-(1.2) $ has at least one positive solution say $ y_0 $ satisfying  $ {r_1} \le \left\| {{y_0}} \right\| \le {r_2} $.
\end{theorem}
\begin{proof} As $ F $ is completely continuous operator and $ F:{\mathcal{K}_0} \to {\mathcal{K}_0} $,
let $ {\Omega _1} = \left\{ {y \in {\mathcal{K}_0}: \left\| y \right\| \ge {r_1}} \right\}$.
Then for any $y \in {\mathcal{K}_0} \cap \partial {\Omega _1}$, we have
\begin{align*}
\left\| {Fy} \right\| & = \mathop {\max }\limits_{t \in {{\left[ {v - 2,v + b} \right]}_{{N_{v - 2}}}}} \lambda \sum\limits_{s = 0}^b {G\left( {t,s} \right)h\left( {s + v - 1} \right)f\left( {s + v - 1,y\left( {s + v - 1} \right)} \right),} \\
& \le \lambda \sum\limits_{s = 0}^b {G\left( {s + v - 1,s} \right)h\left( {s + v - 1} \right)f\left( {s + v - 1,y\left( {s + v - 1} \right)} \right)} \\
& \le \lambda \frac{{\eta {r_1}}}{\lambda }\sum\limits_{s = 0}^b {G(s + v - 1,s)h(s + v - 1)}\,\,\,\,\,\,\,\,\,\,\,\,\,\,\,  from  H1 \\
& = {r_1}
\tag{3.3} \\
& = \left\| y \right\|,
\end{align*}
hence, $ \left\| {Fy} \right\| = \left\| y \right\| $, for $ y \in {\mathcal{K}_0} \cap \partial {\Omega _1} $. \\
Now, let $ {\Omega _2} = \left\{ {y \in {\mathcal{K}_0}:\left\| y \right\| \le {r_2}} \right\} $. Then for any $ y \in {\mathcal{K}_0} \cap \partial {\Omega _2}$ we have,
\begin{align*}
Fy(t) & = \lambda \sum\limits_{s = 0}^b {G\left( {t,s} \right)h\left( {s + v - 1} \right)f\left( {s + v - 1,y\left( {s + v - 1} \right)} \right),}\\
& \ge \lambda \sum\limits_{s = \left[ {\frac{{b + v}}{4} - v + 1} \right]}^{\left[ {\frac{{3(b + v)}}{4} - v + 1} \right]} {G\left( {t,s} \right)h\left( {s + v - 1} \right)f\left( {s + v - 1,y\left( {s + v - 1} \right)} \right)} \\
& \ge \lambda \sum {G(t,s)h(s + v - 1)\frac{{\sigma {r_2}}}{\lambda }} \,\,\,\,\,\,\,\,\,\,\,\,\,\,\,\,\,\,\, from  H2 \\
& = {r_2}
\tag{3.4} \\
& = \left\| y \right\|,
\end{align*}
hence, $ \left\| {Fy} \right\| \ge \left\| y \right\| $, for  $ y \in {\mathcal{K}_0} \cap \partial {\Omega _2} $. \\
So, it follows form theorem $ 2.8 $ that there exists $ y_0 \in {\mathcal{K}_0} $ such that $ {Fy}_0 = y_0 $ i. e., fractional boundary value problem $ (1.1)-(1.2) $ has a positive solution, say $ y_0 $ satisfying  $ {r_1} \le \left\| {{y_0}} \right\| \le {r_2} $.
\end{proof}

In the next theorem we give the existence of at least two positive solutions.

\begin{theorem}\label{}
Assume that $ f $ satisfies condition $ H1 $ and $ H3 $. Then the fractional boundary value problem $ (1.1)-(1.2) $ has at least two positive solutions, say $ y_1 $ and $ y_2 $ such that $0 \le \left\| {{y_1}} \right\| < m < \left\| {{y_2}} \right\|$.
\end{theorem}
\begin{proof} From the assumptions, $ \exists $ $ \varepsilon  > 0 $ and $ r > 0 $ with $ r < m $ s. t. for $ 0 \leq y \leq r $,
$ f\left( {t,y} \right) \ge \frac{{(\sigma  + \varepsilon )}}{\lambda }y $,  \,\,\,\,\,\,\,\,\, $ t \in {\left[ {v - 2,v + b} \right]_{{\mathbb{N}_{v - 2}}}} $. \\
Let $ r_1 \in (0,r) $ and $ \left[ {\frac{{b - v}}{2}} \right] + v \in \left[ {\frac{{b + v}}{4},\frac{{3(b + v)}}{4}} \right] $. \\
Hence, for $ y \in \partial {\Omega _r} $, we have
\begin{align*}
\left( {Fy} \right)\left( {\left[ {\frac{{b - v}}{2}} \right] + v} \right) & = \sum\limits_{s = 0}^b G \left( {\left[ {\frac{{b - v}}{2}} \right] + v,s} \right)\lambda h(s + v - 1) \\
& \times f(s + v - 1,y(s + v - 1))\\
& \ge \lambda \sum\limits_{s = 0}^b G \left( {\left[ {\frac{{b - v}}{2}} \right] + v,s} \right)h(s + v - 1)\frac{{(\sigma  + \varepsilon )}}{\lambda }y \\
& \ge \lambda \frac{{(\sigma  + \varepsilon )}}{\lambda }\left\| y \right\|\sum\limits_{s = \left[ {\frac{{v + b}}{4} - v + 1} \right]}^{s = \left[ {\frac{{3(v + b)}}{4} - v + 1} \right]} {G\left( {\left[ {\frac{{b - v}}{2}} \right] + v,s} \right)} \\
& \times h(s + v - 1) \\
& > \sigma \left\| y \right\|\cdot \frac{1}{\sigma }\\
& = \left\| y \right\| = r.
\tag{3.5}
\end{align*}
Thus, $ \left\| {Fy} \right\| > \left\| y \right\| $, for $ y \in {\mathcal{K}_0} \cap \partial {\Omega _r} $.

On the other hand, suppose $ H3 $ holds, then there exists $ \tau>0 $ and $ {R_1}>0 $ such that  $ f(t,y) \ge \frac{{(\sigma  + \tau )}}{\lambda }y $, $ \forall \,y \ge {R_1} $, $ t \in {\left[ {v - 2,v + b} \right]_{{\mathbb{N}_{v - 2}}}} $. \\ Now let $ R $ such that, $ R > \max \left( {m,{\raise0.7ex\hbox{${{R_1}}$} \!\mathord{\left/
 {\vphantom {{{R_1}} \gamma }}\right.\kern-\nulldelimiterspace}
\!\lower0.7ex\hbox{$\gamma $}}} \right) $ then, we have
\begin{align*}
\left( {Fy} \right)\left( {\left[ {\frac{{b - v}}{2}} \right] + v} \right) & = \sum\limits_{s = 0}^b G \left( {\left[ {\frac{{b - v}}{2}} \right] + v,s} \right)\lambda h(s + v - 1) \\
& \times f(s + v - 1,y(s + v - 1))\\
& \ge \lambda \sum\limits_{s = 0}^b {G\left( {\left[ {\frac{{b - v}}{2}} \right] + v,s} \right)h\left( {s + v - 1} \right)\frac{{(\sigma  + \tau )}}{\lambda }y} \\
& \ge \lambda \frac{{(\sigma  + \tau )}}{\lambda }\left\| y \right\|\sum\limits_{s = \left[ {\frac{{b + v}}{4} - v + 1} \right]}^{\left[ {\frac{{3(b + v)}}{4} - v + 1} \right]} {G\left( {t,s} \right)h\left( {s + v - 1} \right)} \\
& \ge \sigma \left\| y \right\|\cdot \frac{1}{\sigma }\\
& = R,
\tag{3.6}
\end{align*}
hence,  $ \left\| {Fy} \right\| \geq \left\| y \right\| $, for $ y \in {\mathcal{K}_0} \cap \partial {\Omega _R} $.

Now, for any $ y \in \partial {\Omega _m} $, $ H1 $ implies that, $ f\left( {t,y} \right) \le \frac{{\eta m}}{\lambda }$, $ t \in {\left[ {v - 2,v + b} \right]_{{\mathbb{N}_{v - 2}}}} $. \\
Let
\begin{align*}
 Fy(t) & = \lambda \sum\limits_{s = 0}^b {G\left( {t,s} \right)h\left( {s + v - 1} \right)f\left( {s + v - 1,y\left( {s + v - 1} \right)} \right)} \\
& \le \lambda \sum\limits_{s = 0}^b {G\left( {s + v - 1,s} \right)h\left( {s + v - 1} \right)\cdot \frac{{\eta m}}{\lambda }} \\
& = \eta m\cdot \frac{1}{\eta }\\
& = m = \left\| y \right\|,
\tag{3.7}
\end{align*}
hence, $ \left\| {Fy} \right\| \le \left\| y \right\| $, for $ y \in {\mathcal{K}_0} \cap \partial {\Omega _m} $.\\
Therefore from theorem $ 2.8 $ it implies that there are two fixed points $ y_1 $ and $ y_2 $ of operator $ F $ s. t. $ 0 \le \left\| {{y_1}} \right\| < m < \left\| {{y_2}} \right\| $.
\end{proof}

\begin{theorem}\label{}
Assume that, conditions $ H2 $ and $ H4 $ holds, $ f>0 $ for  $ t \in {\left[ {v - 2,v + b} \right]_{{\mathbb{N}_{v - 2}}}} $. Then fractional boundary value problem $ (1.1)-(1.2) $ has at least two positive solutions, say  $ y_1 $ and $ y_2 $ such that $0 \le \left\| {{y_1}} \right\| < m < \left\| {{y_2}} \right\|$.
\end{theorem}
\begin{proof} Suppose that $ H2 $ holds, then there exists $ \varepsilon  > 0 $ $ (\varepsilon  < \eta ) $ and $ 0 < r < m $ such that $ f\left( {t,y} \right) \le \frac{{(\eta  - \varepsilon )}}{\lambda }y $, $ 0\leq y \leq r $,\,\,\,\,\,$ t \in {\left[ {v - 2,v + b} \right]_{{\mathbb{N}_{v - 2}}}} $. \\
Let $ r_1 \in (0,r) $, then for $ y \in \partial {\Omega _{{r_1}}} $, we have
\begin{align*}
Fy(t) & = \lambda \sum\limits_{s = 0}^b {G\left( {t,s} \right)h\left( {s + v - 1} \right)f\left( {s + v - 1,y\left( {s + v - 1} \right)} \right)} \\
& \le \lambda \sum\limits_{s = 0}^b {G\left( {s + v - 1,s} \right)h\left( {s + v - 1} \right)\cdot \frac{{(\eta  - \varepsilon )}}{\lambda }{r_1}} \\
& < \eta {r_1}\sum\limits_{s = 0}^b {G\left( {s + v - 1,s} \right)h\left( {s + v - 1} \right)} \\
& < \eta {r_1}\cdot \frac{1}{\eta }\\
& = {r_1} = \left\| y \right\|,
\tag{3.8}
\end{align*}
hence, we have  $ \left\| {Fy} \right\| < \left\| y \right\| $, for $ y \in \partial {\Omega _{{r_1}}} $.

On the other hand, suppose that $ H4 $ holds, then there exists $ 0 < \tau < \eta  $ and $ R_0 >0 $ s. t. $ f\left( {t,y} \right) \le \tau \eta $, $ y \geq R_0 $,  $ t \in {\left[ {v - 2,v + b} \right]_{{\mathbb{N}_{v - 2}}}} $. \\
Denote $ M = \mathop {\max }\limits_{(t,y) \in {{\left[ {v - 2,v + b} \right]}_{{\mathbb{N}_{v - 2}}}} \times \left[ {0,{R_0}} \right]} f(t,y) $ then \\
$ 0 \le f(t,y) \le \frac{{\left( {\tau y + M} \right)}}{\lambda } $,\,\,\,\, $ 0 \le y <  + \infty $. \\
Let $ {R_2} > \max \left\{ {\frac{M}{{(\eta  - \tau )}},2m} \right\} $. For $ y \in \partial {\Omega _{{R_2}}} $, we have
\begin{align*}
\left\| {Fy} \right\| & = \mathop {\max }\limits_{t \in {{\left[ {v - 2,v + b} \right]}_{{\mathbb{N}_{v - 2}}}}} \lambda \sum\limits_{s = 0}^b {G\left( {t,s} \right)h\left( {s + v - 1} \right)f\left( {s + v - 1,y\left( {s + v - 1} \right)} \right)} \\
& \le \lambda \sum\limits_{s = 0}^b {G\left( {s + v - 1,s} \right)h\left( {s + v - 1} \right)f\left( {s + v - 1,y\left( {s + v - 1} \right)} \right)} \\
& \le \lambda \frac{{\left( {\tau \left\| y \right\| + M} \right)}}{\lambda }\sum\limits_{s = 0}^b {G(s + v - 1,s)h(s + v - 1)} \\
& = \lambda \frac{{\left( {\tau {R_2} + M} \right)}}{\lambda }\cdot \frac{1}{\eta }\\
& < {R_2} = \left\| y \right\|,
\tag{3.9}
\end{align*}
hence, we have  $ \left\| {Fy} \right\| < \left\| y \right\| $, for $ y \in \partial {\Omega _{{R_2}}}. $

Finally, for any $ y \in \partial {\Omega _{m}} $, since $ \gamma m \le y(t) \le m $ for $ t \in \left[ {\frac{{b + v}}{4},\frac{{3\left( {b + v} \right)}}{4}} \right] $, \\ we have
\begin{align*}
\left( {Fy} \right)\left( {\left[ {\frac{{b - v}}{2}} \right] + v} \right) & = \sum\limits_{s = 0}^b G \left( {\left[ {\frac{{b - v}}{2}} \right] + v,s} \right)\lambda h(s + v - 1) \\
& \times f(s + v - 1,y(s + v - 1))\\
& > \lambda \sigma \gamma m\sum\limits_{s = \left[ {\frac{{b + v}}{4} - v + 1} \right]}^{\left[ {\frac{{3(b + v)}}{4} - v + 1} \right]} {G\left( {\left[ {\frac{{b - v}}{2}} \right] + v,s} \right)h\left( {s + v - 1} \right)} \\
& = m = \left\| y \right\|.
\tag{3.10}
\end{align*}
Hence,  $ \left\| {Fy} \right\| > \left\| y \right\| $, for $ y \in {\mathcal{K}_0} \cap \partial {\Omega _m} $. \\
Therefore, by the theorem $ 2.2 $, the proof is complete.
\end{proof}

\begin{example}\label{}
Consider the following fractional boundary value problem,
\[{\Delta ^{\frac{5}{4}}}y\left( t \right) =  - \lambda{\frac{{1}}{100}} {e^{\left( {t + \frac{1}{4}} \right)}}\left( {t + \frac{1}{4}} \right)\left\{ {{y^{\frac{1}{2}}}\left( {t + \frac{1}{4}} \right) + {y^2}\left( {t + \frac{1}{4}} \right)} \right\}\tag{3.11}\]
\[y\left( { - \frac{3}{4}} \right) =0,  y\left( {\frac{{25}}{4}} \right) = 0\tag{3.12}\]
where $ v = \frac{5}{4} $, $ b = 5 $, $ f(t,y) = {\frac{{1}}{100}}t\left( {{y^{\frac{1}{2}}} + {y^2}} \right) $, $ h\left( t \right) = {e^t} $. With a simple computation we can verify that $ \eta > {0.0021} $.
$ f:\left[ {0,\infty } \right) \times \left[ {0,\infty } \right) \to \left[ {0,\infty } \right) $ and $ h:\left[ {0,\infty } \right) \to \left[ {0,\infty } \right) $ and $ f(t,y) $ satisfies the conditions $ H1 $ and $ H3 $, will have at least one positive solution.
\end{example}

\bibliographystyle{plain}

\label{lastpage}
\end{document}